\documentclass[12pt,final]{article}

\usepackage{graphicx}

\usepackage{bbm}

\usepackage{amsmath}
\usepackage{amsthm}
\usepackage{amsfonts}
\usepackage{amssymb}
\usepackage{xcolor}
\usepackage{comment}
\usepackage{tcolorbox}
\usepackage{enumerate}
\usepackage{mdframed}

\newcommand{\eee}{{\rm e}}

\newcommand{\me}{\mathbb{E}}
\newcommand{\mn}{\mathbb{N}}
\newcommand{\mr}{\mathbb{R}}

\DeclareMathOperator{\1}{\mathbbm{1}}

\newcommand{\mmp}{\mathbb{P}}

\newtheorem{thm}{Theorem}[section]
\newtheorem{lemma}[thm]{Lemma}

\newtheorem{cor}[thm]{Corollary}

\newtheorem{assertion}[thm]{Proposition}
\theoremstyle{definition}

\theoremstyle{remark}
\newtheorem{rem}[thm]{Remark}

\begin{document}
\title{On the tails of Dickman-like perpetuities}\date{}
\author{Alexander Iksanov\footnote{Faculty of Computer Science and Cybernetics, Taras Shevchenko National University of Kyiv, Ukraine; e-mail address:
iksan@univ.kiev.ua} \ \ and \ \ Oleh Iksanov\footnote{Faculty of Computer Science and Cybernetics, Taras Shevchenko National University of Kyiv, Ukraine; e-mail address: vivec1x@knu.ua}}
\maketitle

\begin{abstract}
Using a probabilistic technique based on exponential change of measure,
we derive precise tail asymptotics of some perpetuities with distributions close to the Dickman distribution.
\end{abstract}



\section{Introduction}\label{sect:intro}

Let $(M_k, Q_k)_{k\geq 1}$ be independent copies of an $\mr^2$-valued random vector $(M,Q)$ defined on a common probability space $(\Omega, \mathcal{F},  \mmp)$. Whenever the random series $$Q_1+\sum_{k\geq 1}M_1\cdot\ldots\cdot M_kQ_{k+1}$$ converges almost surely (a.s.), its sum that, denoted by -,= is called {\it  perpetuity}. The term `perpetuity' is justified by the fact that such random series
occur in the realm of insurance and finance as sums of discounted payment streams. A theory of perpetuities, accumulated up to 2016, is well documented in the monographs \cite{Buraczewski etal:2016} and \cite{Iksanov:2016}. An incomplete list of more recent publications includes \cite{Bihan+Kolodziejek:2025, Burdzy+Kolodziejek+Tadic:2022, Damek+Kolodziejek:2020, Iksanov+Marynych+Nikitin:2023, Iksanov+Nikitin+Samoilenko:2021, Kevei:2017}. An interesting recent generalization of perpetuities can be found in \cite{Ma+Maillard:2025}.

In Theorem 3.1(a) of \cite{Goldie+Grubel:1996} it is shown that if $|Z|<\infty$ a.s., $M$ and $Q$ are independent, $\mmp\{M\in [0,1]\}=1$, the function $\epsilon \mapsto \epsilon^{-1}\mmp\{M\in [1-\epsilon, 1]\}$ is bounded away from $0$ and $\infty$ for small positive $\epsilon$ and the constant $b:=\sup\{x: \mmp\{Q\geq x\}>0\}$ satisfies $b\in (0,\infty)$, then
\begin{equation}\label{eq:goldie}
-\log \mmp\{Z>t\}~\sim~ b^{-1}t\log t,\quad t\to\infty,
\end{equation}
that is, the right distribution tail of $Z$ exhibits a Poisson-like decay.

We became curious whether a precise, rather than logarithmic, asymptotic can be derived under comparable assumptions. The present paper answers this question in the affirmative. Here are our

\noindent {\bf Standing assumptions}:
\begin{itemize}

\item $M$ and $Q$ are independent;

\item $M$ has a beta distribution with parameters $\alpha>0$ and $1$, that is, $\mmp\{M\in{\rm d}x\}=\alpha x^{\alpha-1}\1_{(0,1)}(x){\rm d}x$;

\item $\me [\log^+|Q|]<\infty$; 

\item $b\in (0,\infty)$.
\end{itemize}

Here, $\log^+ y=\log y$ if $y\geq 1$ and $=0$ if $y\in (0,1)$; $\1_A(x)=1$ if $x\in A$ and $=0$ if $x\notin A$. Whenever $M$ has the beta distribution with parameters $\alpha$ and $1$, we mark $M$, $Z$ and some other letters with the lower index $\alpha$, that is, we write $M_\alpha$ and $Z_\alpha$ in place of $M$ and $Z$. We call the corresponding perpetuity $Z_\alpha$ a {\it Dickman-like} perpetuity. The reason is that the variable $Z_\alpha-1$ arising in the case $\alpha=1$ and $Q=1$ a.s.\ has the {\it Dickman} distribution, whereas the term {\it generalized Dickman distribution} is sometimes used when $\alpha>0$ and $Q=1$. We refer to Theorem 4.7.7 on p.~90 and Section 4.7.8 on p.~93 of \cite{Vervaat:1972} or to a more recent work \cite{Pinsky:2018} for various properties of the generalized Dickman distribution. An even more recent survey pertaining to this distribution can be found in the introduction of \cite{Grahovac etal:2025}.

An extra bonus of our assumptions is a reasonably simple form of the moment generating function of $Z_\alpha$ given in part (b) of the following proposition. We recall that the distribution of a real-valued random variable $X$ is called {\it selfdecomposable} if, for each $c\in (0,1)$, $z\mapsto \me [\eee^{{\rm i}zX}]/\me [\eee^{{\rm i}czX}]$, $z\in\mr$ is the characteristic function of a proper probability distribution. Each selfdecomposable distribution is infinitely divisible with the L\'{e}vy measure $\nu$ given by $\nu({\rm d}x)=|x|^{-1}k(x){\rm d}x$, where $k$ is a nonnegative function which does not increase on $(0,\infty)$ and does not decrease on $(-\infty,0)$, see Corollary 15.11 on p.~95 in \cite{Sato:1999}. According to Example 27.8 on p.~177 in \cite{Sato:1999}, each nondegenerate selfdecomposable distribution on $\mr$ is absolutely continuous with respect to Lebesgue measure.
\begin{assertion}\label{prop:Dickman}
Let $\alpha>0$ and $Z_\alpha$ be the Dickman-like perpetuity, that is, it corresponds to $(M_\alpha,Q)$ satisfying the standing assumptions. Then

\noindent (a) $Z_\alpha-Q_1=\int_{[0,\infty)}\eee^{-y}{\rm d}A(y)$, where $A$ is a compound Poisson process with the jumps $Q_2$, $Q_3,\ldots$ and the jump times $-\log M_{\alpha,\,1}$, $-\log(M_{\alpha,\,1}M_{\alpha,\,2}),\ldots$; in particular, the distribution of $Z_\alpha-Q_1$ is selfdecomposable; its L\'{e}vy measure $\nu^\ast_\alpha$ is given by $\nu^\ast_\alpha({\rm d}x)=|x|^{-1}k_\alpha(x){\rm d}x$, where $k_\alpha(x)=\alpha \mmp\{Q>x\}$ for $x\in (0,\infty)$ and $k_\alpha(x)=\alpha\mmp\{Q\leq x\}$ for $x\in (-\infty, 0)$.

\noindent (b) $\me [\eee^{sZ_\alpha}]<\infty$ for all $s\geq 0$ and
\begin{equation}\label{eq:mgf}
\me [\eee^{sZ_\alpha}]=\me [\eee^{sQ}]\exp\Big(\alpha\int_0^s \frac{\me [\eee^{yQ}]-1}{y}{\rm d}y\Big),\quad s\geq 0.
\end{equation}

\noindent (c) 
For $t>0$,
\begin{equation}\label{eq:dens}
\alpha^{-1}tq_\alpha(t)=\mmp\{Z_\alpha>t\}-\mmp\{Z_\alpha-Q_1>t\},
\end{equation}
where $q_\alpha$ denotes a density, continuous on $(0,\infty)$, of the distribution of $Z_\alpha-Q_1$.

Furthermore,
\begin{equation}\label{eq:asymp1}
\alpha^{-1}tq_\alpha(t)~\sim~\mmp\{Z_\alpha>t\},\quad t\to\infty.
\end{equation}
\end{assertion}
Although parts (a) and (b) of this proposition are by no means new, we did not see them in the literature in the present form. We briefly discuss the proof of Proposition \ref{prop:Dickman} in Section \ref{sect:prop}.

For $s\geq 0$, put $g(s):=\me [\eee^{sQ}]$ and then $$\psi_\alpha(s):=\log \me [\eee^{s(Z_\alpha-Q_1)}]=\alpha \int_0^s \frac{g(y)-1}{y}{\rm d}y.$$ The latter equality is a consequence of \eqref{eq:mgf}. The function $\psi_\alpha$ is strictly convex on $(0,\infty)$, hence $\psi^\prime_\alpha$ is strictly increasing on $(0,\infty)$. Also, $\psi^\prime_\alpha$ is continuous on $(0,\infty)$. We do not assume that $\me [Q]$ is finite. As a consequence, $\lim_{s\to 0+}\psi^\prime_\alpha(s)$ may be equal to $-\infty$. On the other hand, for any $\varepsilon>0$, $\lim_{s\to\infty}\eee^{\varepsilon s}\me [\eee^{-s(b-Q)}]=\infty$. Hence, $s^{-1}g(s)=s^{-1}\eee^{bs}\me [\eee^{-s(b-Q)}]\to\infty$ as $s\to\infty$. In particular, there exists $s^\ast\geq 0$ such that $\psi^\prime_\alpha(s)>0$ for all $s>s^\ast$. For each $t>\psi^\prime_\alpha(s^\ast)$, the equation $\psi^\prime_\alpha(s)
=t$ has then a unique positive solution that we denote by $s_\alpha(t)$. In particular, $$\alpha\frac{g(s_\alpha(t))-1}{s_\alpha(t)}=t.$$ Observe that $\lim_{t\to\infty}s_\alpha(t)=+\infty$.

Put $$p_\alpha(t):=\int_{(-\infty,\,b]}q_\alpha(t-y){\rm d}\mmp\{Q\leq y\},\quad t\in\mr.$$ Then $p_\alpha$ is a density, continuous on $(b,\infty)$, of the distribution of $Z_\alpha$.  
We note in passing that the distribution of $Z_\alpha$ is not necessarily infinitely divisible, let alone selfdecomposable.

As usual, $r_1(t)\sim r_2(t)$ as $t\to\infty$ will mean that $\lim_{t\to\infty}(r_1(t)/r_2(t))=1$. Here is our main result.
\begin{thm}\label{thm:main}
Under the standing assumptions, as $t\to\infty$,
$$p_\alpha(t)~\sim~\frac{1}{\alpha (2\pi b)^{1/2}}t^{1/2}s_\alpha(t) \exp\Big(-ts_\alpha(t)+\alpha \int_0^{s_\alpha(t)}\frac{g(y)-1}{y}{\rm d}y\Big)$$ and
\begin{equation}\label{eq:first}
\mmp\{Z_\alpha>t\}~\sim~\frac{1}{\alpha (2\pi b)^{1/2}}t^{1/2}\exp\Big(-ts_\alpha(t)+\alpha \int_0^{s_\alpha(t)}\frac{g(y)-1}{y}{\rm d}y\Big).
\end{equation}
\end{thm}
\begin{rem}\label{rem:formulae}
Let $X$ be a real-valued random variable with $\mmp\{X>y\}>0$ for all $y>0$ and $\Psi(s):=\log \me [\eee^{sX}]<\infty$ for all $s\geq 0$. Arguing as at the beginning of Section \ref{sect:proofs} it can be shown that
\begin{equation}\label{eq:repr2}
\mmp\{X>t\}=\exp(-ts^\ast(t)+\Psi(s^\ast(t)))J(t)
\end{equation}
for an appropriate function $J$, where, for each $t$ large enough, $s^\ast(t)$ is the unique positive solution to $\Psi^\prime(s)=t$. Put $I(t):=\sup_{s\geq 0}\,(st-\Psi(s))$ for $t>0$. The function $I$ is the Legendre transform of $\Psi$. 
Since $I(t)=ts^\ast(t)-\Psi(s^\ast(t))$ for large $t$, the first factor on the right-hand side of \eqref{eq:repr2} is actually equal to $\exp(-I(t))$ for large $t$.

The main contribution of Theorem \ref{thm:main} is the following universal asymptotic behavior of the second factor on the right-hand side of \eqref{eq:repr2} when $X=Z_\alpha$: $J(t)\sim (s_\alpha(t))^{-1}(2\pi b t)^{-1/2}$ as $t\to\infty$. The first factor in \eqref{eq:repr2} is the standard exponential large-deviation term. In the present setting, it satisfies
\begin{equation}\label{eq:it}
\eee^{-I(t)}~\sim~\alpha^{-1}ts_\alpha(t)\exp\Big(-ts_\alpha(t)+\alpha\int_0^{s_\alpha(t)}\frac{g(y)-1}{y}{\rm d}y\Big),\quad t\to\infty.
\end{equation}
Its form differs slightly from $\exp\big(-ts^\ast(t)+\log \me [\eee^{s^\ast(t)Z_\alpha}]\big)$ because $s_\alpha(t)$ is defined in terms of the moment generating function of $Z_\alpha-Q_1$, whereas $s^\ast(t)$ is defined in terms of that of $Z_\alpha$. The factor $\alpha^{-1}ts_\alpha(t)$ in \eqref{eq:it} comes from $$\me [\eee^{sZ_\alpha}]=g(s)\me [\eee^{s(Z_\alpha-Q_1)}]\quad\text{and}\quad g(s_\alpha(t))=1+\alpha^{-1}ts_\alpha(t)~\sim~\alpha^{-1}ts_\alpha(t).$$ It is straightforward to verify that the two forms of the exponential large-deviation term are asymptotically equivalent. Multiplying \eqref{eq:it} by the above asymptotic formula for $J(t)$, we recover \eqref{eq:first}. More explicit asymptotic formulae for the right-hand side of \eqref{eq:it} are given in Section \ref{sect:examples} for both {\it general} and {\it particular} distributions of $Q$.
%
%
%
%
%
%
\end{rem}

\section{A short survey of relevant works}

The interest in Dickman and generalized Dickman distributions extends well beyond perpetuity theory. Their importance originates in probabilistic number theory, where the Dickman--de Bruijn function $\varrho$ describes the asymptotic density of smooth integers. More precisely, for every fixed $u\geq 1$,
$$\lim_{x\to\infty} x^{-1}\#\{n\leq x:~\text{every prime divisor of}~ n~\text{is at most}~x^{1/u}\}=\varrho(u).$$ Precise estimates for such smooth-number counts play a central role in computational number theory, particularly in the analysis of integer-factorization algorithms, see \cite{Granville:2008} and the references therein. The same function governs the limit behavior of the largest prime factor of a uniformly chosen integer, the longest cycle of a random permutation, and the largest degree of an irreducible factor of a random polynomial over a finite field. These connections fit naturally within the general theory of logarithmic combinatorial structures and are closely related to Poisson--Dirichlet distributions, see \cite{Arratia+Barbour+Tavare:2003}.

The probabilistic Dickman distribution itself also appears in the analysis of algorithms. It arises, for instance, as a limit distribution for the number of comparisons performed by the Quickselect algorithm \cite{Hwang+Tsai:2002}. More generally, Dickman approximation has become a useful testing ground for Stein's method and quantitative distributional approximation \cite{Bhattacharjee+Schulte:2025}, as well as for unusual domain-of-attraction phenomena \cite{Pinsky:2018}. At the process level, the Dickman subordinator arises in renewal theory and marginally relevant disordered systems, including pinning and directed-polymer models \cite{Caravenna+Sun+Zygouras:2019}, while stationary processes with Dickman-type marginals and either short- or long-range dependence have recently been constructed in \cite{Grahovac etal:2025}. Against this background, the present paper shows that the precise tail asymptotics known for the classical Dickman law persist for a substantially broader class of Dickman-like perpetuities, in which the deterministic additive term is replaced by a bounded random variable $Q$. Moreover, this extension is obtained by a probabilistic exponential-change-of-measure argument rather than by the complex-analytic methods traditionally used in the classical case.

We proceed by recalling classical results. Assume that $Q=1$ a.s. According to Proposition \ref{prop:Dickman}, the distribution of $Z_\alpha-1$ admits a density $q_\alpha$ that is continuous on $(0,\infty)$. By using a complex-analytic approach based on the saddle-point method, it was shown in formula (1.6) of \cite{Bruijn:1951} that\footnote{Let $\gamma$ denote the Euler-Mascheroni constant. Actually, de Bruijn worked with the function $t\mapsto \eee^\gamma q_1(t)$ which justifies the appearance of the factor $\eee^\gamma$ in his original formula.}
$$q_1(t)~\sim~ \frac{1}{(2\pi t)^{1/2}}\exp\Big(-\int_0^{s_1(t)}\frac{y\eee^{y}-\eee^y+1}{y}{\rm d}y\Big),\quad t\to\infty,$$ where, for each $t>1$, $s_1(t)$ is the unique positive solution to $s^{-1}(\eee^{s}-1)=t$.
As a consequence, the following asymptotic expansion with $\alpha=1$ was obtained in formula (1.8) of \cite{Bruijn:1951}: as $t\to\infty$,
\begin{equation}\label{eq:verv}
q_\alpha(t)=\exp\Big(-t\Big(\log t+\log\log t-(1+\log \alpha)\Big(1+\frac{1}{\log t}\Big)+\frac{\log\log t}{\log t}+O\Big(\Big(\frac{\log\log t}{\log t}\Big)^2\Big)\Big)\Big).
\end{equation}
To derive the preceding formula with $\alpha>0$ from the formula of de Bruijn, a clever real-analytic argument was worked out in the proof of Lemma 4.7.9 on p.~84 in \cite{Vervaat:1972}.

A streamlined (but still complex-analytic) argument given in the proof of Theorem 8 on p.~374 of \cite{Tenenbaum:1995} shows that 
$$q_1(t)=\frac{1}{(2\pi t)^{1/2}}\exp\Big(-\int_0^{s_1(t)}\frac{y\eee^{y}-\eee^y+1}{y}{\rm d}y\Big)\Big(1+O\Big(\frac{1}{t}\Big)\Big),\quad t\to\infty.$$

Recall for later needs that a positive measurable function $h$ is called regularly varying at $0$ (at $\infty$) of index $\beta\in\mr$ if $\lim_{t\to 0+}(h(\lambda t)/h(t))=\lambda^\beta$ ($\lim_{t\to\infty}(h(\lambda t)/h(t))=\lambda^\beta$) for each $\lambda>0$. A function is called slowly varying at $0$ (at $\infty$) if it is regularly varying at $0$ (at $\infty$) of index $0$.

As has already been stated in the introduction, the present work was inspired by Theorem 3.1(a) in \cite{Goldie+Grubel:1996}. That paper gave impetus to further investigations of logarithmic asymptotics of superexponential  distribution tails of perpetuities \cite{Hitczenko:2010, Hitczenko+Wesolowski:2009, Kolodziejek:2017, Kolodziejek:2018} and logarithmic asymptotics of superexponential distribution tails of fixed points of inhomogeneous smoothing transforms \cite{Alsmeyer+Dyszewski:2017}.

Specifically, suppose that $M$ and $Q$ are nonnegative and independent, and put $$H(t):=-t\log\mmp\{M>1-1/t\},\quad t>1.$$ Assume that $H$ is regularly varying at $\infty$ of index $r>1$, 
and that $b=\sup\{x:\mmp\{Q\geq x\}>0\}\in (0,\infty)$. Then Theorem 3.2 in \cite{Kolodziejek:2018} states that $$-\log \mmp\{Z>t\}~\sim~\Big(\frac{r}{r-1}\Big)^{r-1}H\Big(\frac{t}{b}\Big),\quad t\to\infty.$$ This result does not apply to the Dickman-like perpetuities $Z_\alpha$, since the corresponding function $H$ is regularly varying of index $1$. However, if $Q=b$ a.s.\ and $H$ is ultimately strictly convex and regularly varying of index $1$, then Theorem 4.1 in \cite{Kolodziejek:2017} yields $$-\log \mmp\{Z>t\}~\sim~H\Big(\frac{t}{b}\Big),\quad t\to\infty.$$ This result does apply to $Z_\alpha$, because in this case $H(t)=-t\log (1-(1-1/t)^\alpha)$ for $t>1$, which is ultimately strictly convex and regularly varying at $\infty$ of index $1$.

Recall that, under the standing assumptions, the variable $Z_\alpha-Q_1$ has an infinitely divisible distribution. By Proposition \ref{prop:Dickman}, the supremum of the support of its L\'{e}vy measure $\nu^\ast_\alpha$ is equal to $b$. Hence, by a general result given in Theorem 2(ii) of \cite{Ohkubo:1979}, $$-\log \mmp\{Z_\alpha-Q_1>t\}~\sim~ b^{-1}t\log t,\quad t\to\infty.$$

To close the section, still assuming that $M_\alpha$ has the beta distribution with parameters $\alpha>0$ and $1$, we discuss two other scenarios giving rise to tail behavior different from that in Theorem \ref{thm:main}.

\noindent {\sc Regularly varying tails}. Assume that $\me [\log^+|Q|]<\infty$ and $\mmp\{Q>t\}\sim t^{-\beta}L(t)$ as $t\to\infty$ for some $\beta>0$ and some $L$ slowly varying at $\infty$. The variables $M_\alpha$ and $Q$ may be dependent. According to the Grincevi\v{c}ius-Grey theorem (Theorem 1 in \cite{Grincevicius:1975} and Theorem 1 in \cite{Grey:1994}),
\begin{equation}\label{eq:GrinGrey}
\mmp\{Z_\alpha>t\}~\sim~\frac{1}{1-\me [M_\alpha^\beta]}\mmp\{Q>t\}~\sim~\frac{\alpha+\beta}{\beta}t^{-\beta}L(t),\quad t\to\infty.
\end{equation}
Assume now that $M_\alpha$ and $Q$ are independent. Then the latter limit relation can alternatively be obtained using the fact that the distribution of $Z_\alpha-Q_1$ is infinitely divisible with the L\'{e}vy measure $\nu^\ast_\alpha$ given by $\nu^\ast_\alpha ((x,\infty))=\alpha \int_x^\infty y^{-1}\mmp\{Q>y\}{\rm d}y$ for $x>0$ (this follows along the lines of the proof of Proposition \ref{prop:Dickman}; the assumption that $Q$ is a.s.\ bounded from above appearing in the proposition does not affect the conclusion regarding infinite divisibility, nor the form of the L\'{e}vy measure). By a general result (for instance, Proposition 0 on p.~469 in \cite{Embrechts+Goldie:1981}) relating the tail of an infinitely divisible distribution to the tail of its L\'{e}vy measure, $$\mmp\{Z_\alpha-Q_1>t\}~\sim~\nu^\ast_\alpha ((t,\infty))=\alpha \int_t^\infty y^{-1}\mmp\{Q>y\}{\rm d}y~\sim~ \frac{\alpha}{\beta}t^{-\beta}L(t),\quad t\to\infty.$$ Since the variables $Q_1$ and $Z_\alpha-Q_1$ are independent, and their distribution tails are regularly varying at $\infty$ of the same index, we arrive at \eqref{eq:GrinGrey}.

\noindent {\sc Gamma-like tails}. Assume that $M_\alpha$ and $Q$  are independent, $\me [\log^+ |Q|]<\infty$ and $$
\mmp\{Q>t\}= A\eee^{-ct}+ r(t),\quad t\geq 0$$
for some $A,c> 0$ and a real-valued function $r$ satisfying several technical conditions specified in Theorem 2 of \cite{Buraczewski etal:2018}.
Then, by that theorem,
$$\mmp\{Z_\alpha>t\}~\sim~ Bt^{\alpha A}\eee^{-ct},\quad t\to\infty$$ for a finite constant $B$ explicitly defined in \cite{Buraczewski etal:2018}.

To gain some intuition for this result, it may be helpful to keep in mind the known fact that if $Q$ has the exponential distribution with mean $1/c$, then $Z_\alpha$ has a gamma distribution with parameters $\alpha+1$ and $c$, that is, $\mmp\{Z_\alpha\in {\rm d}x\}= \frac{c^{\alpha+1}}{\Gamma(\alpha+1)}x^\alpha\eee^{-c x}\1_{(0,\infty)}(x){\rm d}x$, where $\Gamma$ is the Euler gamma function. In particular, $$\mmp\{Z_\alpha>t\}~\sim~\frac{c^\alpha}{\Gamma(\alpha+1)}t^\alpha \eee^{-ct},\quad t\to\infty.$$

Let $(X(t))_{t\geq 0}$ be a subordinator, that is, a L\'{e}vy process with nondecreasing paths. Another special class of perpetuities is defined by the integrals $\hat Z:=\int_0^\infty \eee^{-X(u)}{\rm d}u$. It is known that the distribution of $\hat Z$ has a density that we denote by $\hat p$. It follows from the general theory of perpetuities (for instance, Theorems 2.1.7 and 2.1.8 in \cite{Iksanov:2016}) that $\me [\eee^{s\hat Z}]<\infty$ for some $s>0$. In other words, the right distribution tail of $\hat Z$ decays exponentially or superexponentially fast. A precise asymptotic behavior of $\mmp\{\hat Z>t\}$ and $\hat p(t)$ as $t\to\infty$ was derived in \cite{Haas:2025} under the assumption that $\limsup_{s\to\infty}s(-\log \me [\eee^{-sX(1)}])^\prime<1$. The proof is based on a similarity of a functional equation satisfied by $t\mapsto (-\log \mmp\{\hat Z>t\})^\prime$ and a functional equation satisfied by a certain functional of the Laplace exponent of $X(1)$. A similar complex-analytic saddle-point argument was used in \cite{Minchev+Savov:2023} to obtain analogous results, together with precise asymptotics for the derivatives of $\hat p$.

\section{Examples}\label{sect:examples}

First we discuss briefly the asymptotic behavior of $s_\alpha$ and the exponent arising in Theorem \ref{thm:main} for {\it general} distributions of $Q$.

\noindent {\sc Example 1}. Under the standing assumptions,
\begin{equation}\label{eq:asymp1021}
s_\alpha(t)~\sim~b^{-1}\log t \quad \text{and}\quad ts_\alpha(t)-\alpha \int_0^{s_\alpha(t)}\frac{g(y)-1}{y}{\rm d}y~\sim~b^{-1}t\log t,
\quad t\to\infty.
\end{equation}


\noindent {\sc Example 2}. In addition to the standing assumptions, assume that either $\mmp\{b-Q\leq x\}\sim x^\theta \ell(x)$ as $x\to 0+$ for some $\theta\geq 0$ and some $\ell$ slowly varying at $0$ or $\mmp\{Q=b\}>0$ (in which case we put $\theta=0$ in formula \eqref{eq:asymp10211} below). Then
\begin{multline}\label{eq:asymp10211}
s_\alpha(t)~=~b^{-1}(\log t+(1+\theta)\log\log t)+o(\log\log t)\quad\text{and}\\ ts_\alpha(t)-\alpha \int_0^{s_\alpha(t)}\frac{g(y)-1}{y}{\rm d}y~=~b^{-1}t(\log t+(1+\theta)\log\log t)+o(t\log\log t),
\quad t\to\infty.
\end{multline}
We shall prove formulae \eqref{eq:asymp1021} and \eqref{eq:asymp10211} at the end of Section \ref{sect:aux}.

Next, we provide asymptotic formulae for $s_\alpha(t)$, $p_\alpha(t)$ and $\mmp\{Z_\alpha>t\}$ as $t\to\infty$ in terms of elementary functions. 

\noindent {\sc Example 3}. Let $p\in (0,1]$, $\mmp\{Q=b\}=p$ and $\mmp\{Q\leq 0\}=1-p$. Then $s_\alpha(t)$ is the unique positive solution to $\alpha s^{-1}(p\eee^{bs}+\me [\eee^{sQ}\1_{\{Q\leq 0\}}]-1)=t$.
\begin{lemma}
Under the assumptions of Example 3, $s_\alpha(t)$ admits a representation
\begin{multline}
s_\alpha(t)=b^{-1}\Big(\log (\delta t)+\log\log (\delta t)\\+\sum_{n\geq 1}\frac{(-1)^{n+1}}{(\log (\delta t) )^n}\sum_{m=1}^n (-1)^{n+m}\genfrac{[}{]}{0pt}{}{n}{n-m+1}\frac{(\log\log (\delta t))^m}{m!}\Big)+O((t\log t)^{-1}),\label{eq:salpha}
\end{multline}
where the series converges for large $t$, $\delta:=(\alpha b p)^{-1}$, and $(\genfrac{[}{]}{0pt}{}{i}{j})_{i,j\geq 1}$ are the unsigned Stirling numbers of the first kind.
\end{lemma}
\begin{proof}
Put $U(s):=s^{-1}\eee^s$ for $s>0$. For $t>\eee$, denote by $\rho(t)$ the unique solution in $(1,\infty)$ of the equation $U(s)=t$. Theorem 1 in \cite{Jeffrey etal:1995} states that $$\rho(t)=\log t+\log\log t+\sum_{n\geq 1}\frac{(-1)^{n+1}}{(\log t)^n}\sum_{m=1}^n (-1)^{n+m} \genfrac{[}{]}{0pt}{}{n}{n-m+1}\frac{(\log\log t)^m}{m!},$$ where the series converges for all sufficiently large $t$.

Put $h_b(s):=\me [\eee^{sQ/b}\1_{\{Q\leq 0\}}]$ for $s\geq 0$. For all sufficiently large $t$, let $\tilde \rho(t)$ denote the unique solution in $(1,\infty)$ of $s^{-1}(\eee^s-p^{-1}(1-h_b(s)))=t$. Such a solution exists and is unique because the function on the left-hand side is continuous on $[1,\infty)$, ultimately strictly increasing, and diverges to $\infty$. 
Then $s_\alpha(t)=b^{-1}\tilde \rho (t/(\alpha b p))$. We claim that
\begin{equation}\label{eq:inter125}
\tilde \rho(t)-\rho(t)=O((t\log t)^{-1}),\quad t\to\infty,
\end{equation}
which, together with the preceding expansion of $\rho$, proves the lemma.

In view of $h_b(s)\in [0,1]$, we conclude that $$U(\tilde \rho(t))=t+\frac{p^{-1}(1-h_b(\tilde \rho(t)))}{\tilde \rho(t)}\geq t=U(\rho(t)).$$ Since $U$ is strictly increasing on $(1,\infty)$, we have $\tilde \rho(t)\geq \rho(t)$ for all sufficiently large $t$. Using $\rho(t)\sim \log t$ as $t\to\infty$, we infer $$0\leq \frac{p^{-1}(1-h_b(\tilde \rho(t)))}{\tilde \rho(t)}\leq \frac{1}{p\rho(t)}=O\Big(\frac{1}{\log t}\Big),\quad t\to\infty,$$ whence $$U(\tilde \rho(t))-U(\rho(t))=O((\log t)^{-1}),\quad t\to\infty.$$ Moreover, $U^\prime$ is strictly increasing, because $U^{\prime\prime}(s)=s^{-3}\eee^s ((s-1)^2+1)>0$ for $s>0$. By the mean value theorem for differentiable functions, $$U(\tilde \rho(t))-U(\rho(t))\geq U^\prime(\rho(t))(\tilde \rho(t)-\rho(t)).$$ Finally, $$U^\prime(\rho(t))=\frac{\eee^{\rho(t)}}{\rho(t)}\Big(1-\frac{1}{\rho(t)}\Big)=t\Big(1-\frac{1}{\rho(t)}\Big)~\sim~ t,\quad t\to\infty,$$ and \eqref{eq:inter125} follows.

\end{proof}

In principle, formula \eqref{eq:salpha} can be used to obtain an asymptotic expansion of $s_\alpha(t)$ up to the term $O((\log\log t/\log t)^k)$ for any $k\in\mn$. As an example, we give the expansion for $k=5$ (below $\log\log (\delta t)$ always means $(\log\log (\delta t))$):
\begin{multline*}
s_\alpha(t)=b^{-1}\Big(\log (\delta t)+\log\log (\delta t)+\frac{\log\log (\delta t)}{\log (\delta t)}+\frac{\log\log (\delta t) (2-\log\log(\delta t))}{2(\log (\delta t))^2}\\+\frac{\log\log (\delta t) (6-9\log\log(\delta t)+2(\log\log (\delta t))^2)}{6(\log (\delta t))^3}\\+\frac{\log\log (\delta t) (12-36\log\log(\delta t)+22(\log\log (\delta t))^2-3(\log\log (\delta t))^3)}{12(\log (\delta t))^4}+O\Big(\Big(\frac{\log\log t}{\log t}\Big)^5\Big)\Big).
\end{multline*}

Further, for any $k\in\mn$,
\begin{multline*}
\psi_\alpha(s)=\alpha \int_0^s \frac{g(y)-1}{y}{\rm d}y\\=\alpha \Big(p\int_1^s \frac{\eee^{by}}{y}{\rm d}y-\int_1^s \frac{1-\me [\eee^{yQ}\1_{\{Q\leq 0\}}]}{y}{\rm d}y+ \int_0^1 \frac{g(y)-1}{y}{\rm d}y\Big)\\=\alpha p\int_1^s \frac{\eee^{by}}{y}{\rm d}y+O(\log s)=\alpha p \eee^{bs} \sum_{i=1}^k\frac{(i-1)!}{(bs)^i}+O\Big(\frac{\eee^{bs}}{s^{k+1}}\Big),\quad s\to\infty,
\end{multline*}
where the last equality is obtained by repeated integration by parts. By the choice of $s_\alpha$, $$\alpha p \frac{\eee^{bs_\alpha(t)}}{(s_\alpha(t))^i}=\frac{t}{(s_\alpha(t))^{i-1}}+\alpha \frac{1-\me [\eee^{s_\alpha(t)Q}\1_{\{Q\leq 0\}}]}{(s_\alpha(t))^i}=\frac{t}{(s_\alpha(t))^{i-1}}+o(1)$$ as $t\to\infty$, whence $$\psi_\alpha(s_\alpha(t))=t\sum_{i=1}^k \frac{(i-1)!}{b^i(s_\alpha (t))^{i-1}}+O\Big(\frac{t}{(s_\alpha(t))^k}\Big),\quad t\to\infty.$$ Invoking Theorem \ref{thm:main} we conclude that, for any $k\in\mn$, $$p_\alpha(t)=\exp\Big(-t\Big(s_\alpha(t)-\sum_{i=1}^k\frac{(i-1)!}{b^i (s_\alpha(t))^{i-1}}+O\Big(\frac{1}{(s_
\alpha(t))^k}\Big)\Big)\Big),\quad t\to\infty$$ and that $\mmp\{Z_\alpha>t\}$ admits the same asymptotic expansion. The right-hand side can be expressed in terms of elementary functions. However, this task becomes formidable as $k$ grows. For instance, for $k=4$,
\begin{multline*}
p_\alpha(t)=\exp\Big(-b^{-1}t\Big(\log (\delta t)+\log\log (\delta t)-1 
+\frac{\log\log (\delta t)-1}{\log (\delta t)}
\\-\frac{(\log\log (\delta t)-2)^2}{2(\log (\delta t))^2}+\frac{2(\log\log (\delta t))^3-15(\log\log(\delta t))^2+36\log\log(\delta t)-36}{6(\log (\delta t))^3}\\ 
+O\Big(\Big(\frac{\log\log t}{\log t}\Big)^4\Big)\Big)\Big),\quad t\to\infty.
\end{multline*}
Let $b=p=1$, so that $\delta=\alpha^{-1}$. Then the expansion on the right-hand side, truncated with remainder $O((\log\log t/\log t)^2)$, coincides with the expansion in formula \eqref{eq:verv}, as it must.

\noindent {\sc Example 4}. Assume that $Q=b-\eta$, where $\eta$ is a random variable having the gamma distribution with positive parameters $\theta$ and $1/\lambda$. Then $\me [\eee^{sQ}]=\eee^{bs}(1+\lambda s)^{-\theta}$ for $s\geq 0$. In this case, $s_\alpha$ satisfies $$s_\alpha(t)=b^{-1}\big(\log (t/\alpha)+\log (s_\alpha(t)+\alpha/t)+\theta \log(1+\lambda s_\alpha(t))\big).$$ Iterating this once we obtain
\begin{multline*}
s_\alpha(t)=b^{-1}\Big(\log (t/\alpha)+(1+\theta)\log\log (t/\alpha)+\theta\log \lambda+(1+\theta)\log (1/b)+(1+\theta)^2 \frac{\log\log (t/\alpha)}{\log (t/\alpha)}\\+\Big(\frac{\theta b}{\lambda}+(1+\theta)\theta\log\lambda+(1+\theta)^2\log(1/b)\Big)\frac{1}{\log (t/\alpha)}\Big)+O\Big(\Big(\frac{\log\log t}{\log t}\Big)^2\Big),\quad t\to\infty.
\end{multline*}
As in Example 3,
\begin{multline*}
\psi_\alpha(s)=\alpha\int_0^s \frac{g(y)-1}{y}{\rm d}y=\alpha \int_1^s\frac{\eee^{by}}{y(1+\lambda y)^\theta}{\rm d}y+O(\log s)\\=\alpha \eee^{bs}\Big(\frac{1}{bs}\frac{1}{(1+\lambda s)^\theta}+\frac{1}{(bs)^2}\frac{1}{(1+\lambda s)^\theta}+\frac{\theta\lambda}{b^2s}\frac{1}{(1+\lambda s)^{\theta+1}}\Big)+O\Big(\frac{\eee^{bs}}{s^{\theta+3}}\Big)
\end{multline*}
as $s\to\infty$. Further, 
as $t\to\infty$, $$\alpha \frac{\eee^{bs_\alpha(t)}}{s_\alpha(t)(1+\lambda s_\alpha (t))^\theta}=t+o(1),\quad \alpha \frac{\eee^{bs_\alpha(t)}}{s_\alpha(t)(1+\lambda s_\alpha (t))^{\theta+1}}=\frac{t}{\lambda s_\alpha(t)}+O\Big(\frac{t}{(s_\alpha(t))^2}\Big)$$ and $$\alpha \frac{\eee^{bs_\alpha(t)}}{(s_\alpha(t))^2(1+\lambda s_\alpha (t))^\theta}=\frac{t}{s_\alpha(t)}+o(1),$$ whence $$\psi_\alpha(s_\alpha(t))=b^{-1}t\Big(1+\frac{1+\theta}{bs_\alpha(t)}\Big)+O\Big(\frac{t}{(s_\alpha(t))^2}\Big)=b^{-1}t\Big(1+\frac{1+\theta}{\log (t/\alpha)}\Big)+O\Big(\frac{t \log\log t}{(\log t)^2}\Big).$$ Summarizing,
\begin{multline*}
p_\alpha(t)=\exp\Big(-b^{-1}t\Big(
\log (t/\alpha)+(1+\theta)\log\log (t/\alpha)+\theta\log\lambda+(1+\theta)\log (1/b)-1\\+(1+\theta)^2 \frac{\log\log (t/\alpha)}{\log (t/\alpha)}+\Big(\frac{\theta b}{\lambda}+(1+\theta)\theta\log \lambda+(1+\theta)^2\log (1/b)-1-\theta\Big)\frac{1}{\log (t/\alpha)}\\+O\Big(\Big(\frac{\log\log t}{\log t}\Big)^2\Big)\Big)\Big),\quad t\to\infty,
\end{multline*}
and $\mmp\{Z_\alpha>t\}$ admits the same asymptotic expansion.

\section{Auxiliary results}\label{sect:aux}

We start by formulating a part of Lemma 7.1 in \cite{Buraczewski+Iksanov+Marynych:2025}.
\begin{lemma}\label{lem:lst1}
Let $V : \mr \to [0,\infty)$ be a right-continuous and nondecreasing function vanishing on $(-\infty, 0)$. Assume that $\int_{[0,\,\infty)}\eee^{-s^\ast x}{\rm d}V(x)<\infty$ for some $s^\ast>0$. Then $$\lim_{s\to\infty}\frac{\int_{[0,\,\infty)}\eee^{-sx}x{\rm d}V(x)}{\int_{[0,\,\infty)}\eee^{-sx}{\rm d}V(x)}=x_0 := \inf\{x > 0 : V (x) > 0\}.$$
\end{lemma}

For $k\in\mn\cup \{0\}$, $h^{(k)}$ will denote the $k$th order derivative of $h$ with the usual convention that $h^{(0)}=h$. We shall need the following corollary to Lemma \ref{lem:lst1}.
\begin{cor}\label{cor:g}
Suppose $b=\sup\{x: \mmp\{Q\geq x\}>0\}\in (0,\infty)$. Then, for $k=1,2,3$,
$$\lim_{s\to\infty}\frac{g^{(k)}(s)}{g(s)}=b^k.$$ 
\end{cor}
\begin{proof}
Put $f(s):=\me [\eee^{-s(b-Q)}]$ for $s\geq 0$. Fix any $k\in\mn$. Using Lemma \ref{lem:lst1} with $V(x)=\int_{[0,\,x]}y^{k-1}{\rm d}\mmp\{b-Q\leq y\}$ for $x\geq 0$ and noting that then necessarily $x_0=0$ we conclude that
\begin{equation}\label{eq:inter23557}
-\frac{f^{(k)}(s)}{f^{(k-1)}(s)}=\frac{\int_{[0,\infty)}\eee^{-sx}x^k{\rm d}\mmp\{b-Q\leq x\}}{\int_{[0,\infty)}\eee^{-sx}x^{k-1}{\rm d}\mmp\{b-Q\leq x\}}~\to~0,\quad s\to\infty.
\end{equation}

Observe that $g(s)=\eee^{bs}f(s)$ and thereupon
\begin{multline}\label{eq:equal}
\frac{g^\prime(s)}{g(s)}=b+\frac{f^\prime(s)}{f(s)},\quad 
\frac{g^{\prime\prime}(s)}{g(s)}=b^2+2b\frac{f^\prime(s)}{f(s)}+\frac{f^{\prime\prime}(s)}{f(s)}\quad\text{and}\\\frac{g^{(3)}(s)}{g(s)}=b^3+3b^2\frac{f^\prime(s)}{f(s)}+3b\frac{f^{\prime\prime}(s)}{f(s)}+\frac{f^{(3)}(s)}{f(s)}.
\end{multline}
Now the claim of the corollary follows upon invoking formula \eqref{eq:inter23557} with $k=1,2,3$.
\end{proof}

For $s\geq 0$, put $\varphi_\alpha(s):=\log \me [\eee^{sZ_\alpha}]$ and note that $\varphi_\alpha(s)=\log g(s)+\psi_\alpha(s)$.
\begin{lemma}\label{lem:lst}
Under the standing assumptions, for $k=1,2,3$, $$\varphi_\alpha^{(k)}(s)~\sim~\alpha b^{k-1}\frac{g(s)}{s},\quad s\to\infty.$$
\end{lemma}
\begin{proof}
Observe that $$\varphi_\alpha^\prime(s)=\frac{g^\prime(s)}{g(s)}+\psi_\alpha^\prime(s)=\frac{g^\prime(s)}{g(s)}+\alpha \frac{g(s)-1}{s}.$$ By Corollary \ref{cor:g}, $\lim_{s\to\infty}(g^\prime(s)/g(s))=b$, whence
\begin{equation}\label{eq:inter23}
\varphi_\alpha^\prime(s)~\sim~\psi_\alpha^\prime(s)~\sim~ \alpha \frac{g(s)}{s},\quad s\to\infty.
\end{equation}
Further, $$\varphi_\alpha^{\prime\prime}(s)=\frac{g^{\prime \prime}(s)}{g(s)}-\Big(\frac{g^\prime(s)}{g(s)}\Big)^2+\frac{\alpha g^\prime (s)-\psi_\alpha^\prime(s)}{s}.$$ By Corollary \ref{cor:g}, as $s\to\infty$, the sum of the first two terms on the right-hand side vanishes, whereas $g^\prime(s)\sim bg(s)$. This together with \eqref{eq:inter23} proves $$\varphi_\alpha^{\prime\prime}(s)~\sim~\psi_\alpha^{\prime\prime}(s)~\sim~\alpha b \frac{g(s)}{s},\quad s\to\infty.$$ Finally, $$\varphi_\alpha^{(3)}(s)=\frac{g^{(3)}(s)}{g(s)}-3\frac{g^{\prime\prime}(s)g^\prime(s)}{(g(s))^2}+2\Big(\frac{g^\prime(s)}{g(s)}\Big)^3+\frac{\alpha g^{\prime\prime} (s)-2\psi_\alpha^{\prime\prime}(s)}{s},$$ and, by the same reasoning as before, $\varphi_\alpha^{(3)}(s)~\sim~\psi_\alpha^{(3)}(s)~\sim~\alpha b^2 g(s)/s$ as $s\to\infty$.
\end{proof}

To close the section, we prove formulae \eqref{eq:asymp1021} and \eqref{eq:asymp10211} from Section \ref{sect:examples}.

\noindent {\sc Proof of \eqref{eq:asymp1021} and \eqref{eq:asymp10211}}. We start by proving \eqref{eq:asymp1021}. To this end, we use representation $g(s)=\eee^{bs}f(s)$ and assume that the standing assumptions are in force. By the definition of $s_\alpha$, $$\eee^{bs_\alpha(t)}f(s_\alpha(t))=1+\alpha^{-1}ts_\alpha(t)$$ or equivalently
\begin{equation}\label{eq:formula}
s_\alpha(t)=b^{-1}\log (t/\alpha)+b^{-1}(\log(s_\alpha(t)+\alpha/t)-\log f(s_\alpha(t))).
\end{equation}
Then the relation $s_\alpha(t)\sim b^{-1}\log t$ as $t\to\infty$ holds provided we can show that $$\lim_{s\to\infty}s^{-1}\log f(s)=0.$$ The latter is trivial if $\mmp\{Q=b\}>0$ because then $\lim_{s\to\infty}f(s)=\mmp\{Q=b\}>0$. If $\mmp\{Q=b\}=0$, then the desired limit relation is secured by L'H\^{o}pital's rule and \eqref{eq:inter23557} with $k=1$.

Now we show that
\begin{equation}\label{eq:form}
\lim_{t\to\infty}\frac{\alpha\int_0^{s_\alpha(t)}y^{-1}(g(y)-1){\rm d}y}{t}=b^{-1}.
\end{equation}
The second part of \eqref{eq:asymp1021} is then an immediate consequence of this limit relation and the first part of \eqref{eq:asymp1021}. Since the function $y\mapsto y^{-1}(g(y)-1)$ is differentiable on $(0,\infty)$, so is $s_\alpha$. Differentiating $g(s_\alpha(t))=1+\alpha^{-1}ts_\alpha(t)$ we obtain $$g^\prime(s_\alpha(t))s_\alpha^\prime(t)=\alpha^{-1}(s_\alpha(t)+ts^\prime_\alpha(t))$$  or equivalently $$\frac{g^\prime(s_\alpha(t))}{g(s_\alpha(t))}=\frac{s_\alpha(t)+ts^\prime_\alpha(t)}{\alpha g(s_\alpha(t))s^\prime_\alpha(t)}.$$ By Corollary \ref{cor:g} with $k=1$, the left-hand side converges to $b$ as $t\to\infty$. This in combination with
\begin{equation}\label{eq:asymp}
g(s_\alpha(t))~\sim~\alpha^{-1}ts_\alpha(t),\quad t\to\infty
\end{equation}
yields $$\lim_{t\to\infty}ts^\prime_\alpha(t)=b^{-1}$$ and thereupon \eqref{eq:form}. Indeed,
\begin{equation}\label{eq:form2}
\Big(\alpha \int_0^{s_\alpha(t)}\frac{g(y)-1}{y}{\rm d}y\Big)^\prime= \alpha \frac{g(s_\alpha(t))-1}{s_\alpha(t)}s^\prime_\alpha(t)=ts^\prime_\alpha(t),
\end{equation}
and \eqref{eq:form} follows by an application of L'H\^{o}pital's rule in combination with \eqref{eq:form2} and the preceding centered formula.

Next, we treat \eqref{eq:asymp10211}. If $x\mapsto \mmp\{b-Q\leq x\}$ is regularly varying at $0$ of index $\theta$, then, by Theorem 1.7.1' on p.~38 in \cite{BGT:1989}, $f$ is regularly varying at $\infty$ of index $-\theta$.   Then \eqref{eq:formula} entails $$s_\alpha(t)=b^{-1}(\log (t/\alpha)+(1+\theta)\log(s_\alpha(t)))+o(\log (s_\alpha(t))),\quad t\to\infty,$$ and the first part of \eqref{eq:asymp10211} follows.
If $\mmp\{Q=b\}>0$, then, by the reasoning used to prove the first part of \eqref{eq:asymp1021}, the latter representation holds with $\theta=0$, thereby ensuring the first part of \eqref{eq:asymp10211}. The second part of \eqref{eq:asymp10211} follows from \eqref{eq:form} and the first part of \eqref{eq:asymp10211}.

\section{Proof of Theorem \ref{thm:main}}\label{sect:proofs}

We shall prove the asymptotic formula for $p_\alpha(t)$. The asymptotic relation $\mmp\{Z_\alpha>t\}\sim p_\alpha(t)/s_\alpha(t)$ as $t\to\infty$ then follows by L'H\^{o}pital's rule upon noting that, according to \eqref{eq:form2},
\begin{equation*}
\Big(ts_\alpha(t)-\alpha\int_0^{s_\alpha(t)}\frac{g(y)-1}{y}{\rm d}y\Big)^\prime=s_\alpha(t).
\end{equation*}

For each $s > 0$, define a new probability measure $\mmp^{(s)}$
by
\begin{equation}\label{eq:change}
\me^{(s)}[f(Z_\alpha)] =
\frac{\me[\eee^{sZ_\alpha}f(Z_\alpha)]}{\me [\eee^{sZ_\alpha}]},
\end{equation}
where $\me^{(s)}$ denotes expectation with respect to $\mmp^{(s)}$. The equality is assumed to hold for each bounded Borel function $f:\mr\to\mr$. Actually, by a standard argument, the equality holds true for any Borel function $f : \mr\to\mr$ which is not necessarily
bounded, whenever the left- or right-hand side of \eqref{eq:change} is well defined, possibly infinite. An application of formula \eqref{eq:change}  with $f(x)=x$ yields $$\me^{(s)}[Z_\alpha]=\frac{\me [\eee^{sZ_\alpha}Z_\alpha]}{\me [\eee^{sZ_\alpha}]}
=\varphi_\alpha^\prime(s),\quad s>0.$$

Recall that the $\mmp$-distribution of $Z_\alpha$ is absolutely continuous with a density that we denoted by $p_\alpha$. For each $s>0$, the $\mmp^{(s)}$-distribution of $Z_\alpha$ is also absolutely continuous. Denote its density by $p_\alpha^{(s)}$. Then $$p_\alpha^{(s)}(t)=\frac{\eee^{st}p_\alpha(t)}{\me [\eee^{sZ_\alpha}]},\quad t\in\mr,~s>0$$ or equivalently
\begin{equation}\label{eq:change3}
p_\alpha(t)=\eee^{-st}\me [\eee^{sZ_\alpha}]p_\alpha^{(s)}(t),\quad t\in\mr,~s>0.
\end{equation}
Formula \eqref{eq:change3} with $s=s_\alpha(t)$ reads $$p_\alpha(t)=\eee^{-t 
s_\alpha(t)+\varphi_\alpha(s_\alpha(t))} p_\alpha^{(s_\alpha(t))}(t).$$  
Invoking $\eee^{\varphi_\alpha(s)}=g(s)\eee^{\psi_\alpha(s)}$ for $s>0$ and \eqref{eq:asymp} we infer $$\eee^{-ts_\alpha(t)+\varphi_\alpha (s_\alpha(t))}~\sim~ \frac{1
}{\alpha} t s_\alpha(t) \exp\Big(-ts_\alpha(t)+\alpha\int_0^{s_\alpha(t)}\frac{g(y)-1}{y}{\rm d}y\Big),\quad t\to\infty.$$ Thus, it remains to prove that
\begin{equation}\label{eq:asymp3}
p_\alpha^{(s_\alpha(t))}(t)~\sim~\frac{1}{(2\pi b t)^{1/2}},\quad t\to\infty.
\end{equation}

Under $\mmp^{(s_\alpha(t))}$, put $Z_\alpha^{(s_\alpha(t))}:=Z_\alpha-\me^{(s_\alpha(t))}[Z_\alpha]=Z_\alpha-\varphi_\alpha^\prime(s_\alpha(t))$.
We first show that the $\mmp^{(s_\alpha(t))}$-distributions of the variables $t^{-1/2}Z_\alpha^{(s_\alpha(t))}$ converge weakly as $t\to\infty$ to the centered normal distribution with variance $b$. By the L\'{e}vy continuity theorem for characteristic functions this will imply that
\begin{equation}\label{eq:charfunc} \lim_{t\to\infty}\me^{(s_\alpha(t))}\big[\eee^{{\rm i}ut^{-1/2}Z_\alpha^{(s_\alpha(t))}}\big]=\eee^{-bu^2/2},\quad u\in\mr.
\end{equation}
To prove the weak convergence, it is enough to check that, for each $u\in\mr$,
\begin{equation}\label{eq:charfunc2} \lim_{t\to\infty}\me^{(s_\alpha(t))}\big[\eee^{ut^{-1/2}Z_\alpha^{(s_\alpha(t))}}\big]=\eee^{bu^2/2}.
\end{equation}

\noindent {\sc Proof of \eqref{eq:charfunc2}}. Using formula \eqref{eq:change} with $f(x)=\eee^{ut^{-1/2}x}$ we infer
\begin{multline*}
\me^{(s_\alpha(t))}\big[\eee^{ut^{-1/2}Z_\alpha^{(s_\alpha(t))}}\big]=\frac{\me [\eee^{(s_\alpha(t)+ut^{-1/2})Z_\alpha}]}{\me [\eee^{s_\alpha(t) Z_\alpha}]}\eee^{-ut^{-1/2}\varphi_\alpha^\prime(s_\alpha(t))}\\=\exp\big(\varphi_\alpha(s_\alpha(t)+ut^{-1/2})-\varphi_\alpha(s_\alpha(t))-ut^{-1/2}\varphi_\alpha^\prime(s_\alpha(t))\big).
\end{multline*}
By the mean value theorem for twice differentiable functions, there exists $\vartheta_1=\vartheta_1(t,u)\in (0,1)$ such that
\begin{equation*}
\me^{(s_\alpha(t))}\big[\eee^{ut^{-1/2}Z_\alpha^{(s_\alpha(t))}}\big]=\exp\Big(u^2\frac{\varphi_\alpha^{\prime\prime}(s_\alpha(t))}{2t}\Big)\exp\Big(u^2\frac{\varphi_\alpha^{\prime\prime}(s_\alpha(t)+\vartheta_1 ut^{-1/2})-\varphi_\alpha^{\prime\prime}(s_\alpha(t))}{2t}\Big).
\end{equation*}
Lemma \ref{lem:lst} in combination with formula \eqref{eq:asymp} ensures that $$\varphi_\alpha^{\prime\prime}(s_\alpha(t))~\sim~\alpha b g(s_\alpha(t))/s_\alpha(t)~\sim~bt,\quad t\to\infty.$$ Thus, \eqref{eq:charfunc2} follows if we can prove that
\begin{equation}\label{eq:aux1}
\lim_{t\to\infty} \frac{\varphi_\alpha^{\prime\prime}(s_\alpha(t)+\vartheta_1 ut^{-1/2})-\varphi_\alpha^{\prime\prime}(s_\alpha(t))}{t}=0.
\end{equation}
To this end, note that, by the mean value theorem for differentiable functions, there exists $\vartheta_2=\vartheta_2(t,u)\in (0,1)$ such that, for $u\neq 0$,
\begin{multline*}
\varphi_\alpha^{\prime\prime}(s_\alpha(t)+\vartheta_1 ut^{-1/2})-\varphi_\alpha^{\prime\prime}(s_\alpha(t))=\vartheta_1 ut^{-1/2}\varphi_\alpha^{(3)}(s_\alpha(t)+\vartheta_2 ut^{-1/2})\\~\sim~\alpha b^2\vartheta_1 ut^{-1/2}g(s_\alpha(t))/s_\alpha(t)~\sim~ b^2\vartheta_1 ut^{1/2},\quad t\to\infty.
\end{multline*}

Here, the first asymptotic relation follows from Lemma \ref{lem:lst} and Corollary \ref{cor:g}. Indeed, the latter, together with the mean value theorem for differentiable functions applied to $\log g$, ensures that
$\lim_{t\to\infty} (g(s_\alpha(t)+\vartheta_2ut^{-1/2})/g(s_\alpha(t)))=1$. 
The second asymptotic relation is justified by \eqref{eq:asymp}. For $u=0$, the difference on the left-hand side vanishes identically. Hence, in either case, the difference is $o(t)$. The proof of \eqref{eq:aux1} is complete and so is that of \eqref{eq:charfunc2}.

The density $p_\alpha^{(s)}$ is not necessarily bounded or continuous on $\mr$. Therefore, its characteristic function $u\mapsto \me^{(s)}[\eee^{{\rm i}uZ_\alpha}]$ is not necessarily absolutely integrable on $\mr$. For instance, if $Q=1$ a.s., then, according to Theorem 4.7.7(b) in \cite{Vervaat:1972}, $$p_\alpha^{(s)}(t)=\kappa \frac{\eee^{st}(t-1)^{\alpha-1}}{\me [\eee^{sZ_\alpha}]},\quad t\in (1,2)$$ for an explicitly known constant $\kappa>0$. Thus, if $\alpha\in (0,1)$, the function $p_\alpha^{(s)}$ is unbounded at the point $1$. Since we are going to use the Fourier inversion formula, the density $p_\alpha^{(s)}$ has to be smoothed in a standard way described below.

Put $\lambda(u):=(1-|u|)\1_{[-1,\,1]}(u)$ for $u\in\mr$. The function $\lambda$ so defined is the characteristic function of the probability distribution with density $v$ defined by $v(x):=(1-\cos x)/(\pi x^2)$ for $x\in\mr$. For each $c>0$, put $v_c(x):=cv(cx)$ for $x\in\mr$. The corresponding characteristic function is then $\lambda_c(u):=\lambda(u/c)$ for $u\in\mr$. Invoking $$\int_\mr \eee^{{\rm i}ux}\int_\mr p_\alpha^{(s)}(x-y)v_c(y){\rm d}y{\rm d}x=\me^{(s)}[\eee^{{\rm i}uZ_\alpha}]\lambda_c(u)$$ for $u\in\mr$ and $s>0$ the Fourier inversion gives (replacing $s$ with $s_\alpha(t)$) $$\int_\mr p_\alpha^{(s_\alpha(t))}(x-y)v_c(y){\rm d}y=\frac{1}{2\pi}\int_{-c}^c \eee^{-{\rm i}ux}\me^{(s_\alpha(t))}[\eee^{{\rm i}uZ_\alpha}]\lambda_c(u){\rm d}u,\quad x\in\mr.$$ We intend to use this formula with $x=x_t:=t\pm (\log t)^{-1-\varepsilon}$ for some $\varepsilon>0$. Observe that $$t=\psi_\alpha^\prime(s_\alpha(t))=\varphi_\alpha^\prime(s_\alpha(t))-\frac{g^\prime(s_\alpha(t))}{g(s_\alpha(t))}=\me^{(s_\alpha(t))}[Z_\alpha]-b+r_\alpha(t)$$ for some $r_\alpha$ vanishing at $\infty$. The last equality is justified by Corollary \ref{cor:g}. Hence, $x_t=\me^{(s_\alpha(t))}[Z_\alpha]-b+r_\alpha(t)\pm (\log t)^{-1-\varepsilon}$. Making the change of variables $u=wt^{-1/2}$, we obtain
\begin{multline*}
t^{1/2}\int_\mr p_\alpha^{(s_\alpha(t))}(x_t-y)v_c(y){\rm d}y\\=\frac{1}{2\pi}\int_{-ct^{1/2}}^{ct^{1/2}} \eee^{-{\rm i}t^{-1/2}w(-b+r_\alpha(t)\pm (\log t)^{-1-\varepsilon})}\me^{(s_\alpha(t))}\big[\eee^{{\rm i}t^{-1/2}wZ^{(s_\alpha(t))}_\alpha}\big]\lambda_c(t^{-1/2}w){\rm d}w.
\end{multline*}
Put $c_t:=t^{1/2}(\log t)^{2+3\varepsilon}$ with the same $\varepsilon$ as in the definition of $x_t$. We shall prove below that
\begin{equation}\label{eq:inter234}
\lim_{t\to\infty}\int_{-c_tt^{1/2}}^{c_tt^{1/2}} \eee^{-{\rm i}t^{-1/2}w(-b+r_\alpha(t)\pm (\log t)^{-1-\varepsilon})}\me^{(s_\alpha(t))}\big[\eee^{{\rm i}t^{-1/2}wZ^{(s_\alpha(t))}_\alpha}\big]\lambda_{c_t}(t^{-1/2}w){\rm d}w=(2\pi/b)^{1/2}.
\end{equation}
We claim that this entails \eqref{eq:asymp3}. Indeed, according to \eqref{eq:inter234}, $$t^{1/2}\int_\mr p_\alpha^{(s_\alpha(t))}(x_t-y)v_{c_t}(y){\rm d}y=t^{1/2}\int_\mr p_\alpha^{(s_\alpha(t))}(x_t-y/c_t)v(y){\rm d}y~\to~(2\pi b)^{-1/2},\quad t\to\infty.$$ Recall that the distribution of $Z_\alpha-Q_1$ is nondegenerate and selfdecomposable. According to Theorem 1 in \cite{Yamazato:1978}, each nondegenerate selfdecomposable distribution is unimodal. This implies that there exists $x_\ast\geq 0$ such that $q_\alpha$ is nonincreasing on $(x_\ast,\infty)$. We note in passing that differentiating \eqref{eq:dens} in the case $\alpha\in (0,1]$ we conclude that $q_\alpha$ is nonincreasing on $(0,\infty)$.
In view of $p_\alpha(x)=\int_{(-\infty,\,b]}q_\alpha(x-y){\rm d}\mmp\{Q\leq y\}$ we conclude that $p_\alpha$ is nonincreasing on $(x_\ast+b,\infty)$. Write, with $a_t:=t^{1/2}(\log t)^{1+2\varepsilon}$, where $\varepsilon$ is the same as above, $$\int_\mr p_\alpha^{(s_\alpha(t))}(x_t-y/c_t)v(y){\rm d}y=\int_{-a_t}^{a_t} p_\alpha^{(s_\alpha(t))}(x_t-y/c_t)v(y){\rm d}y+\int_{|y|>a_t} p_\alpha^{(s_\alpha(t))}(x_t-y/c_t)v(y){\rm d}y.$$ Using $v(x)\leq 2/(\pi x^2)$ for $x\neq 0$ we conclude that
\begin{multline*}
t^{1/2}\int_{|y|>a_t} p_\alpha^{(s_\alpha(t))}(x_t-y/c_t)v(y){\rm d}y\leq (2t^{1/2}/(\pi a_t^2))\int_\mr p_\alpha^{(s_\alpha(t))}(x_t-y/c_t){\rm d}y\\=2t^{1/2}c_t/(\pi a_t^2)~\to~ 0,\quad t\to\infty.
\end{multline*} Further, 
\begin{multline*}
\int_{-a_t}^{a_t} p_\alpha^{(s_\alpha(t))}(t+(\log t)^{-1-\varepsilon}-y/c_t)v(y){\rm d}y\\=\int_{-a_t}^{a_t}\frac{\eee^{s_\alpha(t)(t+(\log t)^{-1-\varepsilon}-y/c_t)}}{\me [\eee^{s_\alpha (t)Z_\alpha}]} p_\alpha(t+(\log t)^{-1-\varepsilon}-y/c_t)v(y){\rm d}y\\\leq \frac{\eee^{s_\alpha(t)(t+2(\log t)^{-1-\varepsilon})}}{\me [\eee^{s_\alpha (t)Z_\alpha}]}p_\alpha(t)\int_{-a_t}^{a_t}v(y){\rm d}y~\sim~ \frac{\eee^{s_\alpha(t)t}}{\me [\eee^{s_\alpha (t)Z_\alpha}]}p_\alpha(t)=p_\alpha^{(s_\alpha(t))}(t),\quad t\to\infty.
\end{multline*}
We have used $a_t/c_t=(\log t)^{-1-\varepsilon}$ and monotonicity of $p_\alpha$ for large arguments for the inequality. The asymptotic relation $s_\alpha(t)\sim b^{-1}\log t$ as $t\to\infty$ (see \eqref{eq:asymp1021}) entails \newline $\lim_{t\to\infty}s_\alpha(t)(\log t)^{-1-\varepsilon}=0$. This together with the fact that $v$ is a density justifies the asymptotic equivalence. This proves $\liminf_{t\to\infty}t^{1/2}p_\alpha^{(s_\alpha(t))}(t)\geq (2\pi b)^{-1/2}$. The proof of the converse inequality for the upper limit runs similarly:
\begin{multline*}
\int_{-a_t}^{a_t} p_\alpha^{(s_\alpha(t))}(t-(\log t)^{-1-\varepsilon}-y/c_t)v(y){\rm d}y\geq \frac{\eee^{s_\alpha(t)(t-2(\log t)^{-1-\varepsilon})}}{\me [\eee^{s_\alpha (t)Z_\alpha}]}p_\alpha(t)\int_{-a_t}^{a_t}v(y){\rm d}y\\~\sim~ p_\alpha^{(s_\alpha(t))}(t),\quad t\to\infty.
\end{multline*}

\noindent {\sc Proof of \eqref{eq:inter234}.} Observe that $\int_\mr \eee^{-bw^2/2}{\rm d}w=(2\pi/b)^{1/2}$. Hence, it suffices to prove that, for all $A > 0$,
\begin{equation}\label{eq:step1}
\lim_{t\to\infty}\int_{-A}^A\Big|\eee^{-{\rm i}t^{-1/2}w(-b+r_\alpha(t)\pm (\log t)^{-1-\varepsilon})}\me^{(s_\alpha(t))}\big[\eee^{{\rm i}t^{-1/2}wZ^{(s_\alpha(t))}_\alpha}\big]\lambda_{c_t}(t^{-1/2}w)-\eee^{-bw^2/2}\Big|{\rm d}w=0,
\end{equation}
\begin{equation}\label{eq:step2}
\lim_{A\to\infty} \limsup_{t\to\infty} \int_{A<|w|\leq c_t t^{1/2}}\big|\me^{(s_\alpha(t))}[\eee^{{\rm i}t^{-1/2}wZ_\alpha}]\big|{\rm d}w=0
\end{equation}
and that
$\lim_{A\to\infty}\int_{|w|>A}\eee^{-bw^2/2}{\rm d}w=0$. The latter limit relation holds trivially.

We start by proving \eqref{eq:step1}. The function identically equal to $1$ on $\mr$ is the characteristic function of the distribution degenerate at $0$. 
The characteristic functions $w\mapsto \eee^{-{\rm i}t^{-1/2}w(-b+r_\alpha(t) \pm (\log t)^{-1-\varepsilon})} \lambda_{c_t}(t^{-1/2}w)$ converge to $1$ as $t\to\infty$. Using this in combination with \eqref{eq:charfunc} and recalling that the convergence of characteristic functions is locally uniform we arrive at \eqref{eq:step1}.

As a preparation for a proof of \eqref{eq:step2}, write
\begin{multline}
\big|\me^{(s)}[\eee^{{\rm i}u Z_\alpha}]\big|=\Big|\frac{\me [\eee^{(s+{\rm i}u)Q_1}]}{\me [\eee^{sQ_1}]}\Big|\Big|\frac{\me [\eee^{(s+{\rm i}u)(Z_\alpha-Q_1)}]}{\me [\eee^{s(Z_\alpha-Q_1)}]}\Big|\leq \Big|\frac{\me [\eee^{(s+{\rm i}u)(Z_\alpha-Q_1)}]}{\me [\eee^{s(Z_\alpha-Q_1)}]}\Big|\\=\Big|\exp\Big(\alpha\int_0^1 \frac{g((s+{\rm i}u)x)-g(sx)}{x}{\rm d}x\Big)\Big|=\exp\Big(-\alpha\int_0^1 \frac{\me [\eee^{sxQ}(1-\cos (uxQ))]}{x}{\rm d}x\Big)\label{eq:good}\\\leq \exp\Big(-\me\Big[\alpha \int_0^{Q^+} \frac{\eee^{sy}(1-\cos(uy))}{y}{\rm d}y\Big]\Big).
\end{multline}
We shall use the inequality $1-\cos x=2(\sin (x/2))^2\geq 2x^2/\pi^2$ for $x\in [-\pi, \pi]$. Let $|u|\leq \pi/b$. Since $Q\leq b$ a.s., $y\in [0,Q^+]$ entails $uy\in [-\pi, \pi]$ and $1-\cos (uy)\geq 2u^2y^2/\pi^2$. As a consequence,
\begin{multline*}
\big|\me^{(s)}[\eee^{{\rm i}u Z_\alpha}]\big|\leq \exp\Big(-2 \pi^{-2}\alpha u^2\me\Big[\int_0^{Q^+} \eee^{sy}y{\rm d}y\Big]\Big)\\=\exp\Big(-\frac{2\alpha}{\pi^2} \frac{\me\big[
sQ^+\eee^{sQ^+}-\eee^{sQ^+}+1\big]}{s^2}u^2\Big).
\end{multline*}
For a constant $B>0$, $$\int_A^\infty \eee^{-Bu^2}{\rm d}u\leq (2AB)^{-1}\eee^{-BA^2}.$$ As a consequence, with $\tau_\alpha(t):=(s_\alpha(t))^{-2}\me\big[
s_\alpha(t)Q^+\eee^{s_\alpha(t)Q^+}-\eee^{s_\alpha(t)Q^+}+1\big]$,
\begin{multline*}
\int_{A<|w|\leq \pi b^{-1}t^{1/2}}\big|\me^{(s_\alpha(t))}[\eee^{{\rm i}t^{-1/2}wZ_\alpha}]\big|{\rm d}w\leq \int_{|w|>A}\exp\Big(-\frac{2\alpha}{\pi^2}\frac{\tau_\alpha(t)}{t} w^2 \Big){\rm d}w\\\leq \frac{\pi^2 t}{4A\alpha\tau_\alpha(t)}\exp\Big(-\frac{2\alpha}{\pi^2}\frac{\tau_\alpha(t)}{t}A^2\Big).
\end{multline*}
We claim that
\begin{equation}\label{eq:inter235}
\lim_{t\to\infty}\frac{\tau_\alpha(t)}{t}=\frac{b}{\alpha}
\end{equation}
which secures $$\lim_{A\to\infty}\limsup_{t\to\infty}\int_{A<|w|\leq \pi b^{-1}t^{1/2}}\big|\me^{(s_\alpha(t))}[\eee^{{\rm i}t^{-1/2}wZ_\alpha}]\big|{\rm d}w=0.$$

To prove \eqref{eq:inter235}, we first observe that
\begin{equation}\label{eq:inter237}
g(s)=\me [\eee^{sQ^+}]-\mmp\{Q<0\}+\me [\eee^{sQ}\1_{\{Q<0\}}],
\end{equation}
whence $g(s)\sim \me [\eee^{sQ^+}]$ as $s\to\infty$. This together with \eqref{eq:asymp} ensures that $$\frac{\me[\eee^{s_\alpha(t)Q^+}]}{(s_\alpha(t))^2}~\sim~\frac{g(s_\alpha(t))}{(s_\alpha(t))^2}~\sim~\frac{t}{\alpha s_\alpha(t)}=o(t),\quad t\to\infty.$$ Differentiating \eqref{eq:inter237} we also conclude that $g^\prime(s)\sim \me [Q^+\eee^{sQ^+}]$ as $s\to\infty$. Using \eqref{eq:asymp} 
in combination with $g^\prime(s)\sim bg(s)$ as $s\to\infty$ (see Corollary \ref{cor:g}) we infer $$\frac{\me [Q^+\eee^{s_\alpha(t) Q^+}]}{s_\alpha(t)}~\sim~\frac{g^\prime(s_\alpha(t))}{s_\alpha(t)}~\sim~\frac{b}{\alpha}t,\quad t\to\infty,$$ thereby completing the proof of \eqref{eq:inter235}.

It remains to show that
\begin{equation}\label{eq:final}
\lim_{t\to\infty}\int_{\pi b^{-1}t^{1/2}\leq |w|\leq c_tt^{1/2} }\big|\me^{(s_\alpha(t))}[\eee^{{\rm i}t^{-1/2}wZ_\alpha}]\big|{\rm d}w=0.
\end{equation}
It can be checked with the help of either repeated integration by parts or complex integration that, for $a>0$ and $c\in\mr$, $$\int_0^1 \eee^{ax}\cos(cx){\rm d}x=\frac{\eee^a-1}{a}\frac{a(a\cos c+c\sin c)}{a^2+c^2}+\frac{a\cos c+c\sin c-a}{a^2+c^2},$$ whence $$J_{a,\,c}:=\int_0^1 \eee^{ax}(1-\cos(cx)){\rm d}x=\frac{\eee^a-1}{a}\Big(1-\frac{a(a\cos c+c\sin c)}{a^2+c^2}\Big)-\frac{a\cos c+c\sin c-a}{a^2+c^2}.$$ Let $\Theta$ be an angle such that $\cos \Theta=a(a^2+c^2)^{-1/2}$. Then
\begin{multline*}
J_{a,\,c}=\frac{\eee^a-1}{a}\Big(1-\frac{a \cos(\Theta-c)}{(a^2+c^2)^{1/2}}\Big)-\frac{\cos(\Theta-c)-\cos \Theta}{(a^2+c^2)^{1/2}}\\\geq \frac{\eee^a-1}{a}\Big(1-\frac{a}{(a^2+c^2)^{1/2}}\Big)-\frac{2}{(a^2+c^2)^{1/2}}.
\end{multline*}
Recalling the inequality obtained in the second line of \eqref{eq:good} we write, for $|u|\geq \pi/b$,
\begin{multline*}
\int_0^1 \frac{\me [\eee^{sxQ}(1-\cos (uxQ))]}{x}{\rm d}x\geq \me \big[\1_{\{Q^+\geq b/2\}}J_{sQ^+,\,uQ^+}\big]\\\geq \me \Big[\1_{\{Q^+\geq b/2\}}\Big(\frac{\eee^{sQ^+}-1}{sQ^+}\Big(1-\frac{s}{(s^2+u^2)^{1/2}}\Big)-\frac{2}{Q^+(s^2+u^2)^{1/2}}\Big)\Big]\\\geq \me \Big[\1_{\{Q^+\geq b/2\}}\Big(\frac{\eee^{sQ^+}-1}{sQ^+}\Big(1-\frac{s}{(s^2+(\pi/b)^2)^{1/2}}\Big)-\frac{2}{Q^+(s^2+(\pi/b)^2)^{1/2}}\Big)\Big]
\end{multline*}
Using $$1-\frac{s}{(s^2+(\pi/b)^2)^{1/2}}=\frac{(\pi/b)^2}{(s^2+(\pi/b)^2)^{1/2}((s^2+(\pi/b)^2)^{1/2}+s)}\geq \frac{(\pi/b)^2}{2(s^2+(\pi/b)^2)}$$ we obtain
\begin{multline*}
\int_0^1 \frac{\me [\eee^{sxQ}(1-\cos (uxQ))]}{x}{\rm d}x\\\geq \me\Big[\1_{\{Q^+\geq b/2\}}\Big(\frac{\eee^{sQ^+}-1}{sQ^+}\frac{\pi^2}{2b^2(s^2+(\pi/b)^2)}-\frac{2}{Q^+(s^2+(\pi/b)^2)^{1/2}}\Big)\Big]\\\geq \frac{\pi^2}{2b^3s(s^2+(\pi/b)^2)}\me\big[(\eee^{sQ^+}-1)\1_{\{Q^+\geq b/2\}}\big]-\frac{4}{b(s^2+(\pi/b)^2)^{1/2}}.
\end{multline*}
Fix any $d\in (b/2, b)$. By Markov's inequality, $\me [\eee^{sQ^+}]\geq \eee^{ds}\mmp\{Q>d\}$. Thus, $$\frac{\me [\eee^{sQ^+}\1_{\{Q^+<b/2\}}]}{\me [\eee^{sQ^+}]}\leq \frac{\eee^{-(d-b/2)s}}{\mmp\{Q>d\}}~\to~0,\quad s\to\infty$$ and thereupon $$\me [(\eee^{sQ^+}-1)\1_{\{Q^+\geq b/2\}}]~\sim~\me [\eee^{sQ^+}], 
\quad s\to\infty.$$
Recall that $$\me [\eee^{s_\alpha(t)Q^+}]~\sim~\alpha^{-1}ts_\alpha(t),\quad t\to\infty$$ and that, according to \eqref{eq:asymp1021}, $s_\alpha(t)\sim b^{-1}\log t$ as $t\to\infty$. With this at hand, $$\int_0^1 \frac{\me [\eee^{s_\alpha(t)xQ}(1-\cos (uxQ))]}{x}{\rm d}x\geq \Big(\frac{\pi^2}{2\alpha b}+o(1)\Big)\frac{t}{(\log t)^2}-O((\log t)^{-1}),\quad t\to\infty
$$ uniformly in $|u|\geq \pi/b$. Hence,
\begin{multline*}
\int_{\pi b^{-1}t^{1/2}\leq |w|\leq c_tt^{1/2} }\big|\me^{(s_\alpha(t))}[\eee^{{\rm i}t^{-1/2}wZ_\alpha}]\big|{\rm d}w\\\leq \int_{\pi b^{-1}t^{1/2}\leq |w|\leq c_tt^{1/2} }\exp\Big(-\alpha\int_0^1\frac{\me [\eee^{s_\alpha(t)xQ}(1-\cos(t^{-1/2}wxQ))]}{x} {\rm d}x\Big){\rm d}w\\\leq 2c_t t^{1/2}\exp(-C_1 t(\log t)^{-2})~\to~0,\quad t\to\infty
\end{multline*}
for an appropriate constant $C_1>0$. This completes the proof of both \eqref{eq:final} and Theorem \ref{thm:main}.

\begin{rem}
Here, we discuss an alternative proof of the asymptotic formula for \newline $\mmp\{Z_\alpha>t\}$.

By the same reasoning as used in the proof of Theorem \ref{thm:main} it can be shown that $$q_\alpha(t)~\sim~\frac{1}{(2\pi bt)^{1/2}}\exp\Big(-ts_\alpha(t)+\alpha \int_0^{s_\alpha(t)}\frac{g(y)-1}{y}{\rm d}y\Big),\quad t\to\infty.$$ The asymptotic formula for $\mmp\{Z_\alpha>t\}$ then follows from the latter limit relation and \eqref{eq:asymp1}.
\end{rem}

\section{Comments on the proof of Proposition \ref{prop:Dickman}} \label{sect:prop}

(a,b) Fix any $s>0$. By Theorem 3 in \cite{Buraczewski etal:2018}, the assumptions $\mmp\{M_\alpha\in (0,1)\}=1$ and $\me [\eee^{sQ}]<\infty$ entail $\me [\eee^{sZ_\alpha}]<\infty$ and thereupon $\me [\eee^{s(Z_\alpha-Q_1)}]<\infty$.

Put $S_k:=-\log (M_{\alpha,\,1}\cdot\ldots \cdot M_{\alpha,\,k})$ for $k\geq 1$. Then $(S_k)_{k\geq 1}$ is the standard random walk with increments having the exponential distribution of mean $1/\alpha$. Furthermore, $S_1$, $S_2,\ldots$ are the successive jump times of the compound Poisson process $A(t):=\sum_{k\geq 1}Q_{k+1}\1_{\{S_k\leq t\}}$ for $t\geq 0$. The representation $$Z_\alpha-Q_1=\sum_{k\geq 1}M_{\alpha,\,1}\cdot\ldots\cdot M_{\alpha,\,k} Q_{k+1}=\int_{[0,\infty)}\eee^{-y}{\rm d}A(y)$$ is then obvious. As follows from \cite{Jurek+Vervaat:1983}, all the claims of parts (a,b) of the proposition are consequences of this representation. To be more precise, see formula (2.2b) on p.~250 for the L\'{e}vy measure, Section 3 for selfdecomposability and formula (4.4) on p.~255 for a counterpart of $$\log \me [\eee^{s(Z_\alpha-Q_1)}]=\alpha\int_0^s \frac{\me [\eee^{yQ}]-1}{y}{\rm d}y,\quad s\geq 0$$ in terms of characteristic functions (the passage to the moment generating function is straightforward).

\noindent (c) We first note that the right-hand side of \eqref{eq:dens} is nonnegative. This follows from the equality $Z_\alpha-Q_1=M_{\alpha,\,1}Z_\alpha^\prime$, where $Z_\alpha^\prime:=Q_2+M_{\alpha,\,2} Q_3+M_{\alpha,\,2} M_{\alpha,\,3} Q_4+\ldots$ has the same distribution as $Z_\alpha$ and is independent of $M_{\alpha,\,1}$. Using the latter equality we further infer, for $t>0$,
\begin{multline*}
\mmp\{Z_\alpha-Q_1>t\}=\mmp\{M_{\alpha,\,1}Z_\alpha^\prime>t\}=\alpha\int_0^1 \mmp\{Z_\alpha>t/y\}y^{\alpha-1}{\rm d}y\\=\alpha t^\alpha \int_t^\infty v^{-\alpha-1}\mmp\{Z_\alpha>v\}{\rm d}v.
\end{multline*}
Since the function $v\mapsto v^{-\alpha-1}\mmp\{Z_\alpha>v\}$ is continuous on $(0,\infty)$, differentiating the latter equality we obtain $$\alpha^{-1}tq_\alpha(t)=\mmp\{Z_\alpha>t\}-\alpha t^\alpha\int_t^\infty v^{-\alpha-1}\mmp\{Z_\alpha>v\}{\rm d}v= \mmp\{Z_\alpha>t\}-\mmp\{Z_\alpha-Q_1>t\}.$$ Thus, \eqref{eq:dens} does indeed hold.

Fix any $c>1$. To prove \eqref{eq:asymp1}, we first observe that the relation 
$$\lim_{t\to\infty}\frac{\log \mmp\{Z_\alpha>ct\}}{\log \mmp\{Z_\alpha>t\}}=c$$ which holds according to \eqref{eq:goldie} entails $$\lim_{t\to\infty}\frac{\mmp\{Z_\alpha>ct\}}{\mmp\{Z_\alpha>t\}}=0.$$ Using
\begin{multline*}
0\leq \frac{\mmp\{Z_\alpha-Q_1>t\}}{\mmp\{Z_\alpha>t\}}=\frac{\alpha t^\alpha \int_t^{ct}v^{-\alpha-1}\mmp\{Z_\alpha>v\}{\rm d}v+\alpha t^\alpha \int_{ct}^\infty v^{-\alpha-1}\mmp\{Z_\alpha>v\}{\rm d}v}{\mmp\{Z_\alpha>t\}}\\\leq 1-c^{-\alpha}+ c^{-\alpha}\frac{\mmp\{Z_\alpha>ct\}}{\mmp\{Z_\alpha>t\}}
\end{multline*}
and letting $t\to\infty$ and then $c\to 1+$ we conclude that $$\lim_{t\to\infty}\frac{\mmp\{Z_\alpha-Q_1>t\}}{\mmp\{Z_\alpha>t\}}=0.$$ With this at hand, \eqref{eq:asymp1} follows from \eqref{eq:dens}.

\noindent {\bf Acknowledgement}. We are grateful to the two anonymous referees, whose careful reading and helpful comments enabled us to correct several oversights and improve the overall presentation. This work
was supported by the National Research Foundation of Ukraine (project
2023.03/0059 ‘Contribution to modern theory of random series’). A. Iksanov thanks Vitali Wachtel for a useful discussion at an initial stage of the present work.

\end{document}